\documentclass[preprint,sort&compress,12pt]{elsarticle}
\usepackage{amscd,amssymb,amsmath,amsthm}
\usepackage[all]{xy}
\usepackage{array}
\usepackage{etoolbox}
\usepackage{mathtools}

\makeatletter
\def\ps@pprintTitle{%
 \let\@oddhead\@empty
 \let\@evenhead\@empty
 \def\@oddfoot{\reset@font\hfil\thepage\hfil}
 \let\@evenfoot\@oddfoot
}
\makeatother

\newdir{ >}{!/8pt/\dir{}*\dir{>}}

\newtheorem{theorem}{Theorem}[section]

\newtheorem{lemma}[theorem]{Lemma}
\newtheorem{corollary}[theorem]{Corollary}
\newtheorem{prop}[theorem]{Proposition}

\newtheorem*{con}{Conjecture}
\newtheorem{remark}[theorem]{Remark}
\theoremstyle{definition}
\newtheorem{definition}[theorem]{Definition}

\newtheorem*{th25}{Theorem 2.5}
\newtheorem*{th42}{Theorem 4.2}
\newtheorem*{th54}{Theorem 5.4}

\numberwithin{equation}{theorem}

\DeclareMathOperator{\Aut}{\operatorname{Aut}}
\DeclareMathOperator{\Hom}{\operatorname{Hom}}

\DeclareMathOperator{\Ker}{\mathit{ker}}
\DeclareMathOperator{\im}{\mathit{im}}

\DeclareMathOperator{\m}{\mathfrak{m}}
\DeclareMathSymbol{*}{\mathbin}{symbols}{"03}

\journal{}

\begin{document}

\begin{frontmatter}

\title{Non-inner automorphisms of order $p$ in finite $p$-groups of coclass $4$ and $5$}

\author[IISER TVM]{P. Komma\corref{cor1}}
\address[IISER TVM]{School of Mathematics, Indian Institute of Science Education and Research Thiruvananthapuram,\\695551
Kerala, India.}
\ead{patalik16@iisertvm.ac.in}
\cortext[cor1]{Corresponding author. \emph{Phone number}: +91 8606856562.}

\begin{abstract}
A long-standing conjecture asserts that every finite nonabelian $p$-group has a non-inner automorphism of order $p$. This paper proves the conjecture for finite $p$-groups of coclass $4$ and $5$ ($p\ge 5$). We also prove the conjecture for an odd order nonabelian $p$-group $G$ with cyclic center satisfying $C_G(G^p\gamma_3(G))\cap Z_3(G)\le Z(\Phi(G))$.
\end{abstract}


\end{frontmatter}

\section{Introduction}
Let $p$ be a prime number and let $G$ be a finite nonabelian $p$-group. By a celebrated theorem of Gasch\"{u}tz \cite{GW}, $G$ admits a non-inner automorphism of $p$-power order. In 1973, Berkovich \cite[Problem 4.3]{MK} proposed the following conjecture:

\begin{con}
Every finite nonabelian $p$-group admits a non-inner automorphism of order $p$.
\end{con}

This is a simple to state and notoriously hard problem in group theory. The validity of the conjecture for regular $p$-groups follows from a cohomological result of Schmid \cite{PS80} and \cite{DS}. Deaconescu and Silberberg \cite{DS} proved that a finite nonabelian $p$-group $G$ satisfying the condition $C_G(Z(\Phi(G)))\neq \Phi(G)$ has a non-inner automorphism of order $p$ leaving $\Phi(G)$ elementwise fixed. Liebeck \cite{LH} proved that odd order $p$-groups of class $2$ admits a non-inner automorphism of order $p$ leaving $\Phi(G)$ elementwise fixed. Abdollahi \cite{A07,A10} proved the conjecture for $2$-groups of class $2$, $p$-groups for which $G/Z(G)$ is powerful, and $p$-groups of maximal class. Abdollahi, Ghoraishi, and Wilkens \cite{AGW13} proved the conjecture for finite $p$-groups of class $3$, and Abdollahi et al. \cite{AGMG14} proved the conjecture for $p$-groups of coclass $2$. Ruscitti et al. \cite{MLM} proved the conjecture for $p$-groups of coclass $3$ with the exception of $p=3$. Ghoraishi \cite{SG14, SG13} proved the conjecture for groups not satisfying the condition $Z_2^*(G)\le C_G(Z_2^*(G))=\Phi(G)$, and for an odd order $p$-group $G$ for which $(G, Z(G))$ is a Camina pair. Abdollahi and Ghoraishi \cite{AG17} proved the conjecture for $2$-generator finite $p$-groups with abelian Frattini subgroup. Jamali and Viseh \cite{JV13} proved the conjecture for finite $p$-groups with cyclic commutator subgroup. Shabani-Attar \cite{S13} proved the conjectutre for $p$-groups of order $p^m$ and exponent $p^{m-2}$. 

The objectives of this paper are twofold. Firstly we obtain results of independent interest (See Section $2$). In doing this, we look for criteria for the existence of a derivation $\delta: G\rightarrow A$ such that $\delta(\gamma_2(G))\neq 1$ when $G$ is an extra-special group of exponent $p$. Moreover, we prove a structure theorem for finite $p$-groups. 

\begin{th25}
Let $p$ be an odd prime and let $G$ be a finite $p$-group. If $G$ is not powerful, then a normal subgroup $N$ of $G$ exists such that either $G/N$ is an extra-special group of exponent $p$ or $G/N=U\times V$ where $U\le Z(G/N)$ is elementary abelian and $V$ is an extra-special group of exponent $p$.
\end{th25}

Other than giving the structure of a finite $p$-group, Theorem \ref{th:2.5} allows us to construct derivations on every finite $p$-group when the above-indicated criteria are satisfied.

The second objective of this paper is to prove the conjecture for the classes of finite $p$-groups given in the abstract. Section $2$ provides tools for constructing automorphisms of order $p$ in finite $p$-groups with a cyclic center, and the results below are consequences.

\begin{th42}
Let $p$ be an odd prime and let $G$ be a finite nonabelian $p$-group with cyclic center. Suppose that all the automorphisms of $G$ order $p$ fixing $G^p\gamma_3(G)$ elementwise are inner, then the following holds:
\begin{itemize} 
\item [$(i)$] $C_G(G^p\gamma_3(G))\cap Z_3(G)\not\le Z(\Phi(G))$.
\item [$(ii)$] ${{d(G)+1}\choose{2}}\le r$, where $r$ is the coclass of $G$.
\end{itemize}
\end{th42}

\begin{th54}
Let $p\ge 5$ and let $G$ be a finite nonabelian $p$-group.
\begin{itemize}
\item [$(i)$] If $G$ is of coclass $4$, then $G$ admits a non-inner automorphism of order $p$ fixing $G^p\gamma_3(G)$ elementwise.
\item [$(ii)$] If $G$ is of coclass $5$, then $G$ admits a non-inner automorphism of order $p$ fixing $G^p\gamma_4(G)$ elementwise.
\end{itemize}
\end{th54}

The outline of the paper is as follows.

In Section $2$, we study the construction of derivations on finite $p$-groups. 

In Section $3$, we recall some well-known results about the existence of non-inner automorphisms of order $p$ in finite $p$-groups. 
 
Sections $4$ and $5$ are devoted to prove Theorem \ref{th:4.2} and Theorem \ref{th:5.4}, respectively. 

For a finite group $G$, $|G|, \exp(G), Z(G), Z_i(G)$, and $\Phi(G)$ denote the order, the exponent, the center, the $i$-th center, and the Frattini subgroup of $G$. For a finite $p$-group $G$, $d(G)$ and $\Omega_1(G)$ denote the minimal number of generators of $G$ and the subgroup of $G$ generated by all the elements of order $p$ in $G$. 

\section{Finite $p$-groups and derivations}
Let $G$ be a group and let $M$ be a right $G$-module. A derivation $\delta:G\rightarrow M$ is a function such that
\begin{align*}
\delta(gh)=& \delta(g)^h\ \delta(h)\ \text{for\ all}\ g, h\in G.
\end{align*}
\noindent And $\delta$ is a principal derivation if there exists $m\in M$ such that $\delta(g)= m^{-1} m^g$ for all $g\in G$. Let $Z^1(G, M)$ denote the abelian group of all derivations of $G$ to $M$ and $B^1(G, M)$ denote all principal derivations.

Note that the values of a derivation $\delta$ over a set of generators of $G$ will uniquely determine $\delta$. We set up the following notations: Let $F$ be a free group generated by a finite subset $X$ and let $G=\langle X\ |\ r_1, \dots, r_n\rangle$ be a group whose free presentation is $F/R$, where $R$ is the normal closure in $F$ of the set of relations $\{r_1, \dots, r_n\}$ of $G$. Let $\pi: F\rightarrow G$ be the canonical homomorphism. We have that $M$ is a $G$-module if and only if $M$ is an $F$-module on which $R$ acts trivially. Moreover, we have (cf.\ \cite{GKW}):
\begin{lemma}\label{L:2.1}
\begin{itemize}
\item [$(i)$] Let $M$ be an $F$-module. Then every function $f:X\rightarrow M$ extends in a unique way to a derivation $\delta :F\rightarrow M$.

\item [$(ii)$] Let $M$ be a $G$-module and let $\delta:G\rightarrow M$ be a derivation. Then $\bar{\delta}:F\rightarrow M$ given by the composition $\bar{\delta}(f)= \delta(\pi(f))$ is a derivation such that $\bar{\delta}(r_i)=0$ for all $i\in \{1, \dots, n\}$. Conversely, if $\bar{\delta}:F\rightarrow M$ is a derivation such that $\bar{\delta}(r_i)=0$ for all $i\in \{1, \dots, n\}$, then $\delta(fR)= \bar{\delta}(f)$ defines, uniquely, a derivation on $G=F/R$ to $M$ such that $\bar{\delta}= \delta\circ \pi$.
\end{itemize}
\end{lemma}

For a $G$-module $M$, $M^G$ denote the submodule $\{m\in M| m^g=m\ \text{for\ all}\ g\in G\}$ and $[M, G]$ denote the submodule generated by the elements $m^{-1}m^g$ for $m\in M, g\in G$.
\begin{lemma}\label{L:2.2}
Let $p$ be an odd prime and let $G$ be an extra-special group of exponent $p$. Let $M$ be an elementary abelian $p$-group which is also a $G$-module. Suppose that $M^G$ and $[M, G]$ coincides and have order $p$, and that $d(M)\ge d(G)$. Then there exists a derivation $\delta:G\rightarrow M$ such that $\delta(\gamma_2(G))\neq 1$.
\end{lemma}

\begin{proof}
Recall that the order of an extra-special group is $p^{2n+1}$, and for every integer $n\ge 1$ and every odd prime $p$ there is only one isomorphism class for extra-special groups of order $p^{2n+1}$ and exponent $p$ (cf.\ \cite[p.\ 34]{GM}). Thus $d(G)=2n$ and $G$ has a presentation
\begin{align}
\notag & \langle x_1, y_1, \ldots, x_n, y_n, c| [x_i, x_j] = [y_i, y_j]= [x_i, y_j]=1\ \text{for}\ i\neq j,\\
\label{eq:2.2.1}& [x_i, c]= [y_i, c]=1,\ [x_i, y_i]=c,\ x_i^p= y_i^p= c^p=1\ \text{for}\ 1\le i\le n\rangle.
\end{align}

\noindent Let $F$ be a free group on $\{x_1, y_1, \ldots, x_n, y_n, c\}$. Then $M$ is considered as an $F$-module in a natural way. Now we define a map $\delta :\{x_1, y_1, \ldots, x_n, y_n, c\}\rightarrow M$. Let $1\neq z_0\in M^G$. For $i\in \{1, \ldots, n\}$, let $\sigma_i:M\rightarrow (M^G)^{2n-2}$ be given by

\begin{equation*}
\sigma_i(a)= \big((a^{-1})^{x_1}a, (a^{-1})^{y_1}a, \ldots, \widehat{(a^{-1})^{x_i}a}, \widehat{(a^{-1})^{y_i}a}, \ldots, (a^{-1})^{x_n}a, (a^{-1})^{y_n}a\big).
\end{equation*}

\noindent For $x\in F$ and $a, b\in M$, as $M$ is abelian, it follows that $((ab)^{-1})^xab=(a^{-1})^xa(b^{-1})^xb$. Thus $a\mapsto (a^{-1})^xa$ defines a homomorphism $M\rightarrow M^G$, and we get that $\sigma_i$ is a homomorphism. Noting that $|\im(\sigma_i)|\le p^{2n-2}$, $d(M)\ge 2n$ yields that $|\Ker(\sigma_i)|= \frac{|M|}{|\im(\sigma_i)|}\ge p^2$. Since $|M^G|=p\lneq |\Ker(\sigma_i)|$, there exists $a_i\in \Ker(\sigma_i)\setminus M^G$. Note that either $(a_i^{-1})^{x_i}a_i\neq 1$ or $(a_i^{-1})^{y_i}a_i\neq 1$. If $(a_i^{-1})^{x_i}a_i\neq 1$, then we have $M^G=\langle (a_i^{-1})^{x_i}a_i\rangle$. Thus there exists $k\in \{1, \ldots, p-1\}$ such that $z_0=((a_i^{-1})^{x_i}a_i)^k= ((a_i^k)^{-1})^{x_i}a_i^k$. We replace $a_i$ with $a_i^k$ so that $(a_i^{-1})^{x_i}a_i=z_0$, and define $\delta(x_i)=1$, $\delta(y_i)=a_i$. Similarly, if $(a_i^{-1})^{x_i}a_i=1$ and $(a_i^{-1})^{y_i}a_i\neq 1$, we replace $a_i$ with a power of $a_i$ so that $a_i^{-1}a_i^{y_i}=z_0$, and define $\delta(x_i)= a_i$, $\delta(y_i)=1$. Also define $\delta(c)=z_0$. By Lemma \ref{L:2.1} $(i)$, $\delta$ extends to a derivation $F\rightarrow M$. We proceed to check that $\delta$ preserves the relations in \eqref{eq:2.2.1}. First we see that $\gamma_2(F)$ acts trivially on $M$. For $a\in M$ and $x, y\in F$, $a^{xy}=(aa^{-1}a^x)^y=a^ya^{-1}a^x$, in which the last equality holds since $a^{-1}a^x\in M^G$. Similarly, $a^{yx}=a^xa^{-1}a^y=a^{xy}$, and hence $a^{[x, y]}=a$. Now we get an expression for $\delta([x, y])$, $x, y\in F$. Applying $\delta$ to the identity $xy=yx[x, y]$, we obtain $\delta(x)^y\delta(y)=\delta(y)^{x[x, y]}\delta(x)^{[x, y]}\delta([x, y])$. As the action of $[x, y]$ is trivial, we get that $\delta([x, y])= \delta(x)^{-1}\delta(x)^y(\delta(y)^{-1})^x\delta(y)$. Thus, for all $i\in\{1,\ldots, n\}$,
\begin{equation*}
\delta([x_i, y_i])= \delta(x_i)^{-1}\delta(x_i)^{y_i} (\delta(y_i)^{-1})^{x_i}\delta(y_i) =z_0=\delta(c).
\end{equation*}

\noindent Moreover, for all $x\in \{x_i, y_i\}$, $\delta([x, c])=\delta(x)^{-1}\delta(x)^c (z_0^{-1})^xz_0$. Since $\gamma_2(F)$ acts trivially, and the action of relations in \eqref{eq:2.2.1} is trivial, $c$ acts trivially on $M$. Also since $z_0\in M^G$, $(z_0^{-1})^xz_0=1$, and we obtain $\delta([x, c])=1$. Furthermore, as $a_i\in \Ker(\sigma_i)$, $x_j, y_j$ acts trivially on $a_i$ for $i\neq j$. Hence we deduce that $\delta([x,y])=1$ for all $x\in \{x_i, y_i\}$, $y\in \{x_j, y_j| i\neq j\}$. Now it remains to show $\delta(x^p)=1$ for all $x\in \{x_i, y_i, c\}$. Let $\delta(x)=a$. Then $\delta(x^p)=aa^x\cdots a^{x^{p-1}}$. First we show that $a^{x^i}=a(a^{-1}a^x)^i$ for all $i\ge 1$, which is trivially true when $i=1$. Now let $i\ge 1$. By induction hypothesis, $a^{x^{i+1}}=(a(a^{-1}a^x)^i)^x$. Since $a^{-1}a^x\in M^G$, $a^{x^{i+1}}=a^x(a^{-1}a^x)^i=aa^{-1}a^x(a^{-1}a^x)^i$, and the aim holds for $i+1$. Therefore, $\delta(x^p)=a(aa^{-1}a^x)(a(a^{-1}a^x)^2)\cdots (a(a^{-1}a^x)^{p-1})=a^p(a^{-1}a^x)^{{p}\choose{2}}$. We obtain $\delta(x^p)=1$ since $p\mid {{p}\choose{2}}$ as $p\ge 3$ and $M$ is elementary abelian. Hence, by Lemma \ref{L:2.1} $(ii)$, $\delta$ induces a derivation on $G$ which we again denote with $\delta$, and we have $\delta(c)=z_0\neq 1$ as required.
 \end{proof}
 
We prove the following lemma before proving the structure theorem mentioned in the introduction.
\begin{lemma}\label{L:2.3}
Let $p$ be an odd prime and let $G$ be a finite $p$-group with $|\gamma_2(G)|=\exp(G)=p$. Then either $G$ is an extra-special group or there exists $U, V\le G$ such that $G=U\times V$, where $U\le Z(G)$ is elementary abelian, and $V$ is an extra-special group. Furthermore, $G$ has a minimal generating set $\{x_1, y_1, \ldots, x_n, y_n, x_{2n+1}, \ldots, x_{d(G)}\}$, $d(G)\ge 2n$, such that $V=\langle x_1, y_1, \ldots, x_n, y_n\rangle$ and $U=\langle x_{2n+1}\rangle \times\cdots\times \langle x_{d(G)}\rangle$.
\end{lemma}

\begin{proof}
Since $|\gamma_2(G)|=p$ and $\exp(G)=p$, applying \cite[Lemma 4.2]{Ber1} we obtain
\begin{equation}\label{eq:2.3.1}
G=VZ(G),
\end{equation}
\noindent where $V$ is an extra-special group. Note that $Z(V)\le Z(G)$ by \eqref{eq:2.3.1}. If $|Z(G)|=p$, then we get $Z(G)=Z(V)$. Thus $G=VZ(G)=V$, and $G$ is extra-special. Now let $|Z(G)| >p$. As $Z(G)$ is elementary abelian, we can write $Z(G)= Z(V)\times U$. By \eqref{eq:2.3.1} it now follows that $G=UV$. Furthermore, $U\cap V\le U\cap Z(V)=1$, where $U\cap V\le Z(V)$ holds since $U\le Z(G)$. Thus $G=U\times V$. Moreover, $|V|=p^{2n+1}$, and for an odd prime $p$ there is only one isomorphism class for extra-special groups of order $p^{2n+1}$ and exponent $p$ (cf.\ \cite[p.\ 34]{GM}). Hence $d(V)=2n$ and let $V=\langle x_1, y_1, \ldots, x_n, y_n\rangle$. Let $U=\langle x_{2n+1}\rangle\times\cdots\times\langle x_{2n+d(U)}\rangle$. We have that $G=\langle x_1, y_1, \ldots, x_n, y_n, x_{2n+1}, \ldots, x_{2n+d(U)}\rangle$. Since $\exp(G)=p$, $\Phi(G)=\gamma_2(G)$, and so $|G|=p^{d(G)}|\gamma_2(G)|=p^{d(G)}p$. On the other hand, $|G|=|U||V|=p^{d(U)}p^{2n+1}$, and hence we get $d(G)=2n+d(U)$. This completes the proof. 
\end{proof}

We recall the following fact about finite $p$-groups.
\begin{lemma}\label{L:2.4}
Let $G$ be a finite $p$-group and let $N, L$ be normal subgroups of $G$. If $N\le L[N, G]$, then $N\le L$.
\end{lemma}

\begin{theorem}\label{th:2.5}
Let $p$ be an odd prime and let $G$ be a finite $p$-group. If $G$ is not powerful, then a normal subgroup $N$ of $G$ exists such that either $G/N$ is an extra-special group of exponent $p$ or $G/N=U\times V$ where $U\le Z(G/N)$ is elementary abelian and $V$ is an extra-special group of exponent $p$.
\end{theorem}

\begin{proof}
Since $G$ is not powerful, $G^p\gamma_3(G)\lneq \Phi(G)$. Otherwise, we have $\gamma_2(G)\le G^p\gamma_3(G)$, and this yields $\gamma_2(G)\le G^p$ by Lemma \ref{L:2.4}. Now there exists $G^p\gamma_3(G)\le N\le \Phi(G)$ such that $[\Phi(G):N]=p$. Since $N\le \Phi(G)$, and $[\Phi(G), G]\le G^p\gamma_3(G)\le N$, $N\unlhd G$. Clearly $\exp(G/N)=p$. Furthermore, as $G^p\le N$ and $N\lneq \Phi(G)$, $\gamma_2(G)\nleq N$, and hence $\Phi(G)=\gamma_2(G)N$. Thus, $\gamma_2(G/N)= \Phi(G)/N$ and has order $p$. Now the conclusion of the theorem follows by Lemma \ref{L:2.3}.
\end{proof}

\section{Useful results}
 All commutators used in this paper are left-normed and $[g, h]= g^{-1}h^{-1}gh=g^{-1}g^{h}$. We often use the following lemma by Mann \cite{AM}.
\begin{lemma}\label{L:3.1}(Mann)
Let $G$ be a $p$-group of class less than or equal to $p$, and let $x, y\in G$. Then $[x, y^p]=1$ is equivalent to $[x, y]^p=1$ and, similarly, it is equivalent to $[x^p, y]=1$.
\end{lemma}

\begin{corollary}\label{cr:3.2}
Let $G$ be a finite $p$-group and let $t\in Z_p(G)$. Then the following are equivalent.
\begin{itemize}
\item [$(i)$] $t\in C_G(G^p)$.
\item [$(ii)$] $t^p\in Z(G)$.
\item [$(iii)$] $[g, t]^p=1$ for all $g\in G$.
\end{itemize}
\end{corollary}

Let $N$ be a normal subgroup of $G$, then $Z(N)$ can be regarded as a $G/N$-module via conjugation in $G$. Let $C_{\Aut(G)}(G/N, N)$ denote the subgroup of $\Aut(G)$ consisting of all automorphisms $\alpha$ such that $x^{\alpha}=x$ for all $x\in N$ and $g^{-1}g^{\alpha}\in N$ for all $g\in G$. We use the following well-known fact (cf.\ \cite[Satz I.4.4]{BH}):
\begin{prop}\label{P:3.3}
Let $N$ be a normal subgroup of a group $G$, then there is a natural isomorphism $\varphi:Z^1(G/N, Z(N))\rightarrow C_{\Aut(G)}(G/N, N)$ given by $g^{\varphi(f)}=gf(gN)$ for all $g\in G$, $f\in Z^1(G/N, Z(N))$. The image of $B^1(G/N, Z(N))$ under $\varphi$ is the group of inner automorphisms of $G$ induced by elements of $Z(N)$.
\end{prop}

In addition to the above proposition, suppose that $f\in Z^1(G/N, Z(N)\cap Z_i(G))$ and $\varphi(f)=i_u$ is an inner automorphism induced by $u$, then it follows that
\begin{equation}\label{eq:3.3.1}
f(gN)=g^{-1}g^{\varphi(f)}=[g, u]
\end{equation}
\noindent for all $g\in G$. Also $u\in C_G(N)\cap Z_{i+1}(G)$.

The following corollary is given for $N=\Phi(G)$ in \cite[Lemma 3.3]{AGW13}, but for an arbitrary $N$, the proof follows along the same lines, and we include a proof for the benefit of the reader.
\begin{corollary}\label{cr:3.4}
Let $G$ be a finite $p$-group and let $N$ be a normal subgroup of $G$ such that $C_G(N)=Z(N)$. Set $A=\Omega_1(Z(N))$ and $A^*=\{ a\in Z(N)| a^p\in Z(G)\}$. Suppose that all the automorphisms of $G$ of order $p$ fixing $N$ elementwise are inner, then $Z^1(G/N, A\cap Z_i(G)) \cong \frac{A^*\cap Z_{i+1}(G)}{Z(G)}$ for all $i\in \mathbb{N}$. In particular, if $G^p\le N$, then 
\begin{equation}\label{eq:3.4.1}
Z^1(G/N, A\cap Z_i(G))\cong \frac{Z(N)\cap Z_{i+1}(G)}{Z(G)}
\end{equation}
\noindent for all $i\in \{1, \ldots, p-1\}$.
\end{corollary}

\begin{proof}
Since $C_G(N)=Z(N)$, and all the automorphisms of $G$ of order $p$ fixing $N$ elementwise are inner, by Proposition \ref{P:3.3}, we have $Z^1(G/N, A)=B^1(G/N, A^*)$. Moreover, $Z^1(G/N, A\cap Z_i(G))= B^1(G/N, A^*\cap Z_{i+1}(G))$ for all $i\in\mathbb{N}$. Note that $Z(G)\le C_G(N)=Z(N)$. Hence $Z(G)\le A^*$ holds trivially. Thus $(A^*\cap Z_{i+1}(G))^{G/N}= Z(G)$, and by the fact that $B^1(H, M)\cong \frac{M}{M^H}$ for a $H$-module $M$, we get $B^1(G/N, A^*\cap Z_{i+1}(G))\cong \frac{A^*\cap Z_{i+1}(G)}{Z(G)}$. In addition, if $G^p\le N$, then $Z(N)\le C_G(G^p)$, and hence $(Z(N)\cap Z_p(G))^p\le Z(G)$ by Corollary \ref{cr:3.2}, and so $Z(N)\cap Z_p(G)\le A^*$. Therefore, $A^*\cap Z_{i+1}(G)= Z(N)\cap Z_{i+1}(G)$ for all $i\in \{1, \ldots, p-1\}$.
\end{proof}

Recall that for a maximal subgroup $\m$ of a finite $p$-group $G$, we either have $Z(\m)\le Z(G)$ or $C_G(Z(\m))=\m$.

We now collect some facts, which gives a reduction to the conjecture.
\begin{lemma}\label{L:3.5}
Let $G$ be a finite nonabelian $p$-group. Then $G$ admits a non-inner automorphism of order $p$ fixing $\Phi(G)$ elementwise, if one of the following occurs:
\begin{itemize}
\item [$(i)$] The nilpotency class of $G$ is $2$ or $3$ and $p\ge 3$ (\cite[Theorem 1]{LH}, \cite[Theorem\ 4.4]{AGW13}).
\item [$(ii)$] $G/Z(G)$ is powerful (\cite[Theorem\ 2.6]{A10}).
\item [$(iii)$] $C_G(Z(\Phi(G)))\neq \Phi(G)$ \cite{DS}.
\item [$(iv)$] $G$ is regular \cite{PS80, DS}.
\item [$(v)$] $Z(\m)\le Z(G)$ for a maximal subgroup $\m$ of $G$ (\cite[Lemma\ 9.108]{RJJ}).
\item [$(vii)$] $d\big(Z_2(G)/Z(G)\big)\neq d\big(Z(G)\big) d(G)$ (\cite[Corollary\ 2.3]{A10}).
\item [$(viii)$] $d\big(\Omega_1(Z_2(G))\big)< d\big(Z(G)\big) d(G)$ (\cite[Remark\ 2]{AG17}).
\end{itemize}
\end{lemma}

\begin{lemma}\label{L:3.6}
Let $1\neq N$ be a proper normal subgroup of $G$. If $C_G(Z(N))=N$, then $Z(G)\lneq Z(N)=C_G(N)$.
\end{lemma}

\begin{proof}
As $Z(N)\le N$, $C_G(N)\le C_G(Z(N))=N$, and hence $C_G(N)=Z(N)$. Furthermore, $Z(G)\le C_G(N)=Z(N)$. Since $C_G(Z(N))=N\lneq G$, we have $Z(N)\not\le Z(G)$.
\end{proof}

\begin{remark}\label{R:3.7}
Let $p$ be a prime and let $G$ be a finite nonabelian $p$-group. Let $Z_2^*(G)=\{a\in Z_2(G)| a^p\in Z(G)\}$. Then $Z_2^*(G)\le C_G(G^p)$ by Corollary \ref{cr:3.2}, and $[Z_2^*(G), \gamma_2(G)]=1$ holds trivially. Thus $Z_2^*(G)\le C_G(\Phi(G))$.
\end{remark}

\begin{remark}\label{R:3.8}
Let $p$ be an odd prime and let $G$ be a finite nonabelian $p$-group. Suppose that all the automorphisms of $G$ of order $p$ fixing $\Phi(G)$ elementwise are inner. Then $\Omega_1(Z_2(G))\le C_G(\Phi(G))=Z(\Phi(G))$. Furthermore, setting $G=G/\Phi(G)$, $A=C=\Omega_1(Z_2(G))$, and $D=\Omega_1(Z(G))$ in \cite[Lemma\ 2.3]{AGW13} yields that
\begin{equation*}
d\big(Z^1(G/\Phi(G), \Omega_1(Z_2(G)))\big)\ge d\big(\Omega_1(Z_2(G))\big)d(G)- d\big(\Omega_1(Z(G))\big) {{d(G)}\choose{2}}.
\end{equation*}
\noindent By \eqref{eq:3.4.1}, $Z^1(G/\Phi(G), \Omega_1(Z_2(G)))\cong \frac{Z(\Phi(G))\cap Z_3(G)}{Z(G)}$, and hence the above gives a lower bound for $d\bigg(\frac{Z(\Phi(G))\cap Z_3(G)}{Z(G)}\bigg)$.
\end{remark}

Any of the conditions of Lemma \ref{L:3.5} yields the existence of a non-inner automorphism of order $p$ fixing $G^p\gamma_3(G)\le \Phi(G)$ elementwise. In addition, we have the following lemma when $d(G)=2$.
\begin{lemma}\label{L:3.9}(\cite[Theorem\ 1]{AG17})
Let $p$ be an odd prime and let $G$ be a $2$-generator finite $p$-group. If $G$ fails to satisfy the condition $Z(\Phi(G))\lneq Z(G^p\gamma_3(G))= C_G(G^p\gamma_3(G))$ or if $d\bigg(\frac{Z(G^p\gamma_3(G))\cap Z_3(G)}{Z(G)}\bigg)< 2d\big(\Omega_1(Z_2(G))\big)$, then $G$ admits a non-inner automorphism of order $p$ fixing $G^p\gamma_3(G)$ elementwise.
\end{lemma}

\begin{remark}\label{R:3.10}
Let $p$ be an odd prime and let $G$ be a finite nonabelian $p$-group such that $Z(G)\le Z(G^p\gamma_3(G))$. Since $[Z(G^p\gamma_3(G)), G^p]=1$, $(Z(G^p\gamma_3(G))\cap Z_3(G))^p\le Z(G)$ by Corollary \ref{cr:3.2}. Moreover, $|\Omega_1(A)|=\big|\frac{A}{A^p}\big|$ for every finite abelian $p$-group $A$, so we obtain $\big|\Omega_1(Z(G^p\gamma_3(G))\cap Z_3(G))\big| \ge \big|\frac{Z(G^p\gamma_3(G))\cap Z_3(G)}{Z(G)}\big|$. Both $\Omega_1(Z(G^p\gamma_3(G))\cap Z_3(G))$ and $\frac{Z(G^p\gamma_3(G))\cap Z_3(G)}{Z(G)}$ are elementary abelian, and thus $d\big(\Omega_1(Z_3(G^p\gamma_3(G))\cap Z_3(G))\big)\ge d\bigg(\frac{Z(G^p\gamma_3(G))\cap Z_3(G)}{Z(G)}\bigg)$.
\end{remark}

We rewrite these reductions as a hypothesis.
\begin{definition} 
We say that a finite nonabelian $p$-group $G$ of odd order satisfies Hypothesis $A$, if the following holds true for $G$.
\begin{itemize}
\item [$(i)$] The nilpotency class of $G$ is at least $4$.
\item [$(ii)$] $G/Z(G)$ is not powerful.
\item [$(iii)$] $C_G(Z(\Phi(G))=\Phi(G)$ and $Z(G)\lneq Z(\Phi(G))=C_G(\Phi(G))$.
\item [$(iv)$] $G$ is not regular.
\item [$(v)$] $C_G(Z(\m))=\m$ and $Z(G)\lneq Z(\m)= C_G(\m)$ for every maximal subgroup $\m$ of $G$.
\item [$(vi)$] $\Omega_1(Z_2(G))\le Z_2^*(G)\le Z(\Phi(G))$.
\item [$(vii)$] $G$ satisfies \eqref{eq:3.11.1}$-$\eqref{eq:3.11.3}.
\begin{align}
\label{eq:3.11.1} d\big(Z_2(G)/Z(G)\big)&= d\big(Z(G)\big) d(G).\\
\label{eq:3.11.2} d\big(\Omega_1(Z_2(G))\big)&\ge d\big(Z(G)\big) d(G).\\
\label{eq:3.11.3} d\bigg(\frac{Z(\Phi(G))\cap Z_3(G)}{Z(G)}\bigg)&\ge d\big(\Omega_1(Z_2(G))\big)d(G)- d\big(\Omega_1(Z(G))\big) {{d(G)}\choose{2}}.
\end{align}
\item [$(viii)$] Either $d(G)\ge 3$ or \eqref{eq:3.11.4}$-$\eqref{eq:3.11.6} holds for $G$: 
\begin{align}
\label{eq:3.11.4} Z(\Phi(G))\lneq Z(G^p\gamma_3(G))&= C_G(G^p\gamma_3(G)).\\
\label{eq:3.11.5} d\bigg(\frac{Z(G^p\gamma_3(G))\cap Z_3(G)}{Z(G)}\bigg)&\ge 2d\big(\Omega_1(Z_2(G))\big).\\
\label{eq:3.11.6} d\big(\Omega_1(Z(G^p\gamma_3(G))\cap Z_3(G))\big)&\ge 2d\big(\Omega_1(Z_2(G))\big).
\end{align}
\end{itemize}
\end{definition}

Let $p$ be an odd prime and let $G$ be a finite nonabelian $p$-group. If $d(G)\ge 3$ and $G$ does not have a non-inner automorphism of order $p$ fixing $\Phi(G)$ elementwise or if $d(G)=2$ and $G$ does not have a non-inner automorphism of order $p$ fixing $G^p\gamma_3(G)$ elementwise, then $G$ satisfies Hypothesis $A$.

\section{Finite $p$-groups having cyclic center}
This section proves the conjecture for every finite nonabelian $p$-group $G$ with cyclic center satisfying $C_G(G^p\gamma_3(G))\cap Z_3(G)\le Z(\Phi(G))$. We begin by proving the following lemma.

\begin{lemma}\label{L:4.1}
Let $p$ be an odd prime and let $G$ be a finite nonabelian $p$-group with cyclic center. Then $\Omega_1(Z_2(G))\le Z(G^p\gamma_3(G))$, if all the automorphisms of $G$ of order $p$ fixing $\Phi(G)$ elementwise are inner. \end{lemma}

\begin{proof}
We assume Hypothesis $A$ for $G$. Thus $\Omega_1(Z_2(G))\le Z(\Phi(G))$, and so $[G^p\gamma_3(G), \Omega_1(Z_2(G))]=1$. Now we proceed to show that $\Omega_1(Z_2(G))\le G^p\gamma_3(G)$. Let $a\in \Omega_1(Z_2(G))$. For all $g\in G$, $a^g= a[a, g]$, and since $a^p=1$, we get $[a, g]^p=1$ by Corollary \ref{cr:3.2}, and so $[a, g]\in \Omega_1(Z(G))$. Since $|\Omega_1(Z(G))|=p$, we obtain that the number of conjugates of $a$ in $G$ is at most $p$. Thus, either $a\in \Omega_1(Z(G))$, or $[G:C_G(a)]=p$, and $a\in Z(C_G(a))$ in the latter case. Furthermore, if $a\in\Omega_1(Z(G))$, then $a\in Z(\m)$ for every maximal subgroup $\m$ of $G$ by Hypothesis $A$. Thus in either case, $a\in Z(\m)$ for a maximal subgroup $\m$ of $G$. Let $g\in G\setminus \m$. Since $a\in \Omega_1(Z_2(G))$, $aa^g\cdots a^{g^{p-1}}= a^p[a, g]^{{p}\choose{2}}=1$, and so a derivation $\delta:G/\m\rightarrow Z(\m)$ exists with $\delta(g\m)=a$. Note that the order of $\delta$ is $p$. Let $\alpha= \varphi(\delta)\in C_{\Aut(G)}(G/\m, \m)$. Then $\alpha$ has order $p$, and fixes $\Phi(G)\le \m$ elementwise. Hence $\alpha=i_u$ is an inner automorphism of $G$. It now follows that $u\in C_G(\Phi(G))=Z(\Phi(G))$, and $a=[g, u]$ by \eqref{eq:3.3.1}. Thus $a\in [G, \Phi(G)]\le G^p\gamma_3(G)$.
\end{proof}

Below we recall a couple of well-known commutator identities that are often used in this paper. For $x, y, z$, elements of a group $G$, we have
\begin{align}
\label{eq:4.1.1} [xy, z]&= [x, z]^y\ [y, z]= [x, z] [x, z, y] [x, y]\ \text{and}\\
\label{eq:4.1.2} [z, xy]&= [z, y] [z, x]^y= [z, y] [z, x] [z, x, y].
\end{align}

Abdollahi and Ghoraishi \cite[Theorem\ 1]{AG17} proved that if a $2$-generator finite $p$-group $G$ of odd order does not have a non-inner automorphism of order $p$ fixing $G^p\gamma_3(G)$ elementwise, then $Z(\Phi(G))\lneq C_G(G^p\gamma_3(G))=Z(G^p\gamma_3(G))$. Abdollahi \cite[Theorem\ 2.5]{A10} proved that if a finite nonabelian $p$-group $G$ of coclass $r$ does not have a non-inner automorphism of order $p$ fixing $\Phi(G)$ elementwise, then $d\big(Z(G)\big)\big(d(G)+1\big)\le r +1$. In the next theorem, we obtain similar reductions to the conjecture when $Z(G)$ is cyclic.

\begin{theorem}\label{th:4.2}
Let $p$ be an odd prime and let $G$ be a finite nonabelian $p$-group with cyclic center. Suppose that all the automorphisms of $G$ order $p$ fixing $G^p\gamma_3(G)$ elementwise are inner, then the following holds:
\begin{itemize} 
\item [$(i)$] $C_G(G^p\gamma_3(G))\cap Z_3(G)\not\le Z(\Phi(G))$.
\item [$(ii)$] ${{d(G)+1}\choose{2}}\le r$, where $r$ is the coclass of $G$.
\end{itemize}
\end{theorem}

\begin{proof}
We assume that $G$ satisfies Hypothesis $A$. We prove $(i)$ by showing the existence of an automorphism of order $p$ that fixes $G^p\gamma_3(G)$ and $G/\Omega_1(Z_2(G))$ elementwise, but not $\Phi(G)$, and we prove $(ii)$ by using the conditions of Hypothesis $A$ and by $(i)$. 
\begin{itemize}
\item [$(i)$] By Hypothesis $A$, $G$ is not powerful. Hence, by Theorem \ref{th:2.5}, $G$ has a normal subgroup $G^p\gamma_3(G)\le N\lneq \Phi(G)$ such that $G/N=U/N\times V/N$, where $V/N$ is an extra-special group and $U/N\le Z(G/N)$ is elementary abelian. Furthermore, we assume $\{x_1, y_1, \dots, x_n, y_n, x_{2n+1}, \dots, x_{d(G)}\}$, $d(G)\ge 2n$, is a minimal generating set for $G$ such that $V/N=\langle \bar{x_1}, \dots, \bar{y_n}\rangle$ and $U/N=\langle \overline{x_{2n+1}}\rangle \times \cdots \times \langle \overline{x_{d(G)}}\rangle$. Since $N\le \Phi(G)$ and $\Omega_1(Z_2(G))\le Z(\Phi(G))$, we obtain $[\Omega_1(Z_2(G)), N]=1$, and $\Omega_1(Z_2(G))\le G^p\gamma_3(G)\le N$ by Lemma \ref{L:4.1}. Thus $\Omega_1(Z_2(G))\le Z(N)$ which is a $G/N$-module. Set $M=C_{\Omega_1(Z_2(G))}(U)\le Z(N)$. Since $U/N\le Z(G/N)$, $U/N\unlhd G/N$, and $U\unlhd G$. Thus $C_G(U)\unlhd G$, and hence $M=C_G(U)\cap \Omega_1(Z_2(G))\unlhd G$. Therefore, $M$ is a $G/N$-module. Next we check that the conditions of Lemma \ref{L:2.2} holds when $M$ is considered as a $V/N$-module. Since $U/N$ acts trivially on $M$, we obtain $M^{V/N}=M^{G/N}=\Omega_1(Z(G))$. In particular, $|M^{V/N}|=p$. Now we look for a comparison of $d(M)$ and $d(V/N)$. If $G/N=V/N$, then $M=\Omega_1(Z_2(G))$, and we deduce that $d\big(\Omega_1(Z_2(G))\big)\ge d(G)= d(G/N)$ by \eqref{eq:3.11.2}. Now let $G/N\gneq V/N$. Consider the map $\sigma:\Omega_1(Z_2(G)) \rightarrow \Omega_1(Z(G))^{d(G)-2n}$ given by
\begin{align*}
\sigma(a)=\big([x_{2n+1}, a], \ldots, [x_{d(G)}, a]\big).
\end{align*}
\noindent Let $x\in G$. By Corollary \ref{cr:3.2}, $[x, a]^p=1$ for all $a\in \Omega_1(Z_2(G))$, and expanding $[x, ab]$ using \eqref{eq:4.1.2}, we obtain that $a\mapsto [x, a]$ defines a homomorphism $\Omega_1(Z_2(G))\rightarrow \Omega_1(Z(G))$. Thus $\sigma$ is a homomorphism and $\im(\sigma)\le \Omega_1(Z(G))^{d(G)-2n}$. Note that $\Ker(\sigma)=M$. Hence $|\Omega_1(Z_2(G))|=|M||\im(\sigma)|\le |M|p^{d(G)-2n}$, and $p^{d(G)}\le |\Omega_1(Z_2(G))|$ by \eqref{eq:3.11.2}, so that $p^{2n}\le |M|$. Thus $2n=d(V/N)\le d(M)$, because $M$ is elementary abelian. Moreover, $2n\le d(M)$ implies that $M\not\le \Omega_1(Z(G))$, so we get $[M, V/N]=\Omega_1(Z(G))=M^{V/N}$. Therefore, applying Lemma \ref{L:2.2} with $G=V/N$, we obtain a derivation $\delta\in Z^1(V/N, M)$ with $\delta([x_1, y_1])\neq 1$, and $\delta$ has order $p$. If $G/N=V/N$, then $\delta\in Z^1(G/N, M)$, otherwise the extension $\delta'\in Z^1(G/N, M)$ of $\delta$ that corresponds to $(1_{U/N}, \delta)\in \Hom(U/N, M)\times Z^1(V/N, M)$ (See \cite[Lemma\ 1.2]{CG06}) will have order $p$, and satisfies $\delta'([x_1, y_1])\neq 1$, and we denote $\delta'$ with $\delta$ in the latter case. Let $\alpha= \varphi(\delta)\in C_{\Aut(G)}(G/N, N)$. Then $\alpha$ has order $p$, and fixes $G^p\gamma_3(G)\le N$ elementwise. Hence $\alpha=i_u$ is an inner automorphism of $G$. It now follows that $u\in C_G(G^p\gamma_3(G))\cap Z_3(G)$, and $[x_1, y_1, u]= \delta([x_1, y_1])\ne 1$ by \eqref{eq:3.3.1}, so that $u\not\in C_G(\Phi(G))= Z(\Phi(G))$. 
 
\item [$(ii)$] Let $|G|=p^n$. By Hypothesis $A$, $n-r\ge 4$. First note that $\frac{G}{Z_{n-r-1}(G)}$ is not cyclic, and so $\bigg|\frac{G}{Z_{n-r-1}(G)}\bigg|\ge p^2$, and $\bigg|\frac{Z_{n-r-1}(G)}{Z_3(G)}\bigg|\ge p^{n-r-4}$. Now we obtain a lower bound for $|Z_3(G)/Z(G)|$. Since $Z(G)$ is cyclic, we get $d\big(\Omega_1(Z_2(G))\big)\ge d(G)$ by \eqref{eq:3.11.2}, and using this in \eqref{eq:3.11.3} yields
\begin{align}\label{eq:4.2.1}
d\bigg(\frac{Z(\Phi(G))\cap Z_3(G)}{Z(G)}\bigg)\ge d(G)^2-{{d(G)}\choose{2}}= {{d(G)+1}\choose{2}}.
\end{align}
\noindent Furthermore, $Z_3(G)\not\le Z(\Phi(G))$ by $(i)$ so that $\bigg|\frac{Z_3(G)}{Z(\Phi(G))\cap Z_3(G)}\bigg|\ge p$. Hence $\bigg|\frac{Z_3(G)}{Z(G)}\bigg|\ge p^{{{d(G)+1}\choose{2}}+1}$ by \eqref{eq:4.2.1}. Thus
 \begin{align*}
 p^n=&\big|Z(G)\big| \bigg|\frac{Z_3(G)}{Z(G)}\bigg| \bigg|\frac{Z_{n-r-1}(G)}{Z_3(G)}\bigg| \bigg|\frac{G}{Z_{n-r-1}(G)}\bigg|\\
 \ge&p\ p^{{{d(G)+1}\choose{2}}+1}\ p^{n-r-4}\ p^2,
 \end{align*}
\noindent yielding ${{d(G)+1}\choose{2}}\le r$.
\end{itemize}
\end{proof}

Let $G$ be a finite nonabelian $p$-group, $p\ge 3$. Suppose that $Z(G)$ is cyclic and all the automorphisms of $G$ of order $p$ leaving $G^p\gamma_3(G)$ elementwise fixed are inner, then 
\begin{equation}\label{eq:4.2.2}
C_G(G^p\gamma_3(G))\cap Z_3(G)\not\leq Z(\Phi(G)),
\end{equation}
\noindent In next definition we write this as a hypothesis.

\begin{definition}
A finite nonabelian $p$-group $G$ of odd order satisfies Hypothesis $B$, if
\begin{itemize}
\item [$(i)$] Hypothesis $A$ holds for $G$ and
\item [$(ii)$] either $Z(G)$ is not cyclic, or \eqref{eq:4.2.2} holds for $G$.
\end{itemize}
\end{definition}
 
The homomorphisms like $\sigma$ considered in the proof of Theorem \ref{th:4.2} appear several times in the next section. To avoid the repetition of arguments, we will record the following lemma.

\begin{lemma}\label{L:4.4}
Let $G$ be a finite $p$-group, $N\unlhd G$, and let $H$ be a group with a homomorphism $\chi:H\rightarrow G/N$. Let $M\le Z(N)\cap Z_n(G)$ be a normal abelian subgroup of $G$ considered as a $H$-module via $\chi$. Let $x'_1, \ldots, x'_l\in H$, then the map $\lambda:M\rightarrow Z(N)\cap Z_{n-w+1}(G)$ given by $\lambda(a)=c(x'_1, \ldots, x'_l, a)$ is a homomorphism, where $c$ is a commutator of weight $w\ge 2$ in $\{x'_1, \ldots, x'_l, a\}$ and of weight $1$ in $a$. Furthermore, if $n\le p$, then the image of $M^*$ under $\lambda$ is contained in $\Omega_1(Z(N)\cap Z_{n-w+1}(G))$, where $M^*=\{a\in M| a^p\in Z(G)\}$. 
\end{lemma}

\begin{proof}
We have that $c(x'_1, \ldots, x'_l, a)= c( x_1, \dots, x_l, a)$, where $x_k\in G$ satisfy $\chi(x'_k)= x_kN$ for all $k\in \{1, \ldots, l\}$. Thus $\im(\lambda)\le Z(N)\cap Z_{n-w+1}(G)$, and expanding $c(x'_1, \ldots, x'_l, ab)$ for all $a, b\in M$, we deduce that $\lambda$ is a homomorphism. Furthermore, if $n\le p$, then $M\le Z_p(G)$ so that $\lambda(M^*)\le \Omega_1(Z(N))$ by Corollary \ref{cr:3.2}. Thus we obtain $\lambda(M^*)\le \Omega_1(Z(N)\cap Z_{n-w+1}(G))$. 
\end{proof}

The following technical lemma will be useful in Section $5$.
\begin{lemma}\label{L:4.5}
Let $G$ and $H$ be two groups, and let $N\unlhd G$. Let $\chi: H\rightarrow G/N$ be a homomorphism. Let $M\le Z(N)\cap Z_4(G)$ be a normal abelian subgroup of $G$ considered as a $H$-module via $\chi$ and let $\delta:H\rightarrow M$ be a derivation. Let $x', y', z', w'\in H$ and let $x, y, z, w\in G$ such that $\chi(x')=xN, \chi(y')=yN, \chi(z')=zN, \chi(w')=wN$. If $\delta(x')=a_1$, $\delta(y')=a_2$, and $\delta(z')=a_3$, then we have the following.
\begin{itemize}
\item [$(i)$] $\delta([y', x'])=[a_2, x] [y, a_1] [y, x, a_2][y, x, a_1] [y, x, [a_1, y]]$.
\item [$(ii)$] $\delta([y', x', z'])= [a_2, x, z] [y, a_1, z] [y, x, a_2, z] [y, x, a_1, z] [y, x, a_3] [y, x, z, a_3]$.
\item [$(iii)$] if $M\le C_G(\gamma_3(G))$, then $\delta([y', x', z', w'])= [\delta([y', x', z']), w]$.
\item [$(iv)$] if $G$ is a $p$-group, $p\ge 5$, and $a_1^p=1$, then $\delta((x')^p)=1$.
\item [$(v)$] if $G$ is a $3$-group, $a_1\in Z_3(G)$ such that $a_1^3=1$, and $[a_1, x, x]=1$, then $\delta((x')^3)=1$.
\end{itemize}
\end{lemma}

\begin{proof}
For $a\in M$ and $h\in H$ with $\chi(h)=gN$, we have $a^h= g^{-1}ag$. Thus applying $\delta$ to the identity $y'x'=x'y'[y', x']$, we obtain
\begin{equation*}
a_2^xa_1= a_1^{y[y, x]} a_2^{[y, x]} \delta([y', x']).
\end{equation*}
\noindent Writing $a_1^{y[y, x]}= a_1[a_1, y[y, x]]$, and expanding $[a_1, y[y, x]]$ by \eqref{eq:4.1.2}, we get 
\begin{equation*}
a_2[a_2, x]a_1= a_1[a_1, [y, x]] [a_1, y] [a_1, y, [y, x]] a_2[a_2, [y, x]] \delta([y', x']),
\end{equation*}
\noindent which yields $(i)$. Next to prove $(ii)$, we apply $(i)$ with $x'= z'$ and $y'= [y', x']$. Since $M\le Z_4(G)$, we get $[y, x, z, [a_3, [y, x]]]=1$, and since $\delta([y', x'])\in Z_3(G)$ by $(i)$, we get $[y, x, z, \delta([y', x'])]=1$. Thus 
\begin{equation*}
\delta([y', x', z'])= [\delta([y', x']), z] [y, x, a_3] [y, x, z, a_3].
\end{equation*}
\noindent Using $(i)$ in $[\delta([y', x']), z]$ yields $[[a_2, x] [y, a_1] [y, x, a_2] [y, x, a_1] [y, x, [a_1, y]], z]$, and expanding this by the repeated use of \eqref{eq:4.1.1} gives $(ii)$. Similarly to prove $(iii)$, we apply $(i)$ with $x'=w'$ and $y'=[y', x', z']$. Since $M\le Z_4(G)$, we obtain $\delta([y', x', z', w']= [\delta([y', x', z']), w] [y, x, z, \delta(w')]$, in which $[y, x, z, \delta(w')]=1$ since $[\gamma_3(G), M]=1$ by the assumption in $(iii)$. To prove $(iv)$ and $(v)$, let us first express $\delta((x')^p)$ as
\begin{align}\label{eq:4.5.1} 
\delta((x')^p)= a_1a_1^x\cdots a_1^{x^{p-1}}= a_1^p[a_1, x]^{{p}\choose{2}} [a_1, x, x]^{{p}\choose{3}} [a_1, x, x, x]^{{p}\choose{4}}.
\end{align}
\noindent When $p\ge 5$, since $a_1^p=1$ and $a_1\in Z_4(G)\le Z_p(G)$, we obtain $[a_1, x]^p=[a_1, x, x]^p=[a_1, x, x, x]^p=1$ by Corollary \ref{cr:3.2}. Moreover, $p$ divides ${{p}\choose{i}}$, $i=2, 3, 4$, hence we get $\delta((x')^p)=1$ by \eqref{eq:4.5.1}. In order to prove $(v)$, we deduce that $\delta((x')^3)= a_1^3[a_1, x]^3$ by \eqref{eq:4.5.1}. Since $a_1^3=1$ and $a_1\in Z_3(G)$, we get $[a_1, x]^3=1$ by Corollary \ref{cr:3.2}. This yields $\delta((x')^3)=1$.
\end{proof}

\section{Existence of a non-inner automorphism of order $p$ in finite $p$-groups of coclass $4$ and $5$}
In this section we prove the conjecture for finite nonabelian $p$-groups of coclass $4$ and $5$, $p\ge 5$. Suppose that $G$ is a finite nonabelian $p$-group of order $p^n$, class $c$, and coclass $r$. Since $\big|\frac{G}{Z_{c-1}(G)}\big|\ge p^2$, and $\big|\frac{Z_{c-1}(G)}{Z_i(G)}\big|\ge p^{c-1-i}$, we have $\big|\frac{G}{Z_i(G)}\big|\ge p^{c+1-i}$, and thus $|Z_i(G)|\le p^{i+r-1}$ for all $i\in \{1, \ldots, c-1\}$. In particular, for finite $p$-groups of coclass $4$ and $5$, we have $|Z_i(G)|\le p^{i+4}$ for all $i\in\{1, \ldots, c-1\}$.

\begin{theorem}\label{th:5.1}
Let $p$ be an odd prime and let $G$ be a finite nonabelian $p$-group of class $c$ such that $|Z_i(G)|\le p^{i+4}$ for all $i\in \{1, \ldots, c-1\}$. Then $G$ admits a non-inner automorphism of order $p$ fixing $G^p\gamma_3(G)$ elementwise, if one of the following occurs:
\begin{itemize}
\item [$(i)$] $Z(G)$ is not cyclic.
\item [$(ii)$] $d(G)\ge 3$.
\item [$(iii)$] $|\Omega_1(Z_2(G))|\ge p^3$.
\end{itemize}
\end{theorem}

\begin{proof}
Suppose that $G$ does not satisfy the conclusion of the theorem, then we can assume Hypothesis $B$ for $G$. Furthermore, since $|Z_2(G)|\le p^6$, we get $\bigg|\frac{Z_2^*(G)}{Z(G)}\bigg|\le p^5$, and hence we obtain \eqref{eq:5.1.1} by \eqref{eq:3.11.1}: 
\begin{equation}\label{eq:5.1.1}
d(G)d\big(Z(G)\big)\le 5.
\end{equation}
\begin{itemize}
\item [$(i)$] Since $G$ is nonabelian, $d(G)\ge 2$. Thus, if $Z(G)$ is not cyclic, then we get $d(G)= d\big(Z(G)\big)= 2$ by \eqref{eq:5.1.1}. Using $d(G)=d\big(Z(G)\big)=2$ in \eqref{eq:3.11.2} yields $d\big(\Omega_1(Z_2(G))\big)\ge 4$, and using this in \eqref{eq:3.11.3} gives that $d\bigg(\frac{Z(\Phi(G))\cap Z_3(G)}{Z(G)}\bigg)\ge 6$. Thus $|Z( \Phi(G))\cap Z_3(G)|\ge p^6|Z(G)|\ge p^8$, where we get the second inequality since $d\big(Z(G)\big)=2$. This is a contradiction to $|Z_3(G)|\le p^7$, whence the proof.

\item [$(ii)$] We now assume $Z(G)$ is cyclic by $(i)$. Then using that $\bigg|\frac{Z_3(G)}{Z(G)}\bigg|\ge p^{{{d(G)+1}\choose{2}}+1}$, which we obtained in the proof of Theorem \ref{th:4.2} $(ii)$, we get $\bigg|\frac{Z_3(G)}{Z(G)}\bigg|\ge p^7$. This yields $|Z_3(G)|\ge p^8$, a contradiction to $|Z_3(G)|\le p^7$.

\item [$(iii)$] We now assume that $Z(G)$ is cyclic and $d(G)=2$ by $(i)$ and $(ii)$. Since $\Omega_1(Z_2(G))$ is elementary abelian, $|\Omega_1(Z_2(G))|\ge p^3$ implies that $d\big(\Omega_1(Z_2(G))\big)\ge 3$, and thus we obtain \eqref{eq:5.1.2} by \eqref{eq:3.11.6}:
\begin{equation}\label{eq:5.1.2}
d\big(\Omega_1(Z(G^p\gamma_3(G)\big)\cap Z_3(G)))\ge 6.
\end{equation}
\noindent In the next few lines, we aim to find the isomorphism class of $G/G^p\gamma_3(G)$. First note that $G/G^p\gamma_3(G)$ is a nonabelian group. Otherwise, we have $\gamma_2(G)\le G^p\gamma_3(G)$, which yields that $G$ is powerful by Lemma \ref{L:2.4}. Now let $H=\langle h_1,h_2\rangle$ be a finite $p$-group of class $2$ and of exponent $p$. Then it follows that $\gamma_2(H)=\langle [h_1, h_2]\rangle$, and since $\exp(H)=p$, we get $|\gamma_2(H)|=p$. Furthermore, since the class of $H$ is $2$, we have $\gamma_2(H)\le Z(H)$, and since $\exp(H)=p$, we have $\gamma_2(H)=\Phi(H)$, and so we get $[H: Z(H)]\le [H:\gamma_2(H)]=p^2$. Note that $H/Z(H)$ is not cyclic. Thus $[H: Z(H)]=p^2$, and we get $\gamma_2(H)=Z(H)$. Therefore, $H$ is an extra-special group of order $p^3$ and of exponent $p$. Taking $H=G/G^p\gamma_3(G)$ in the above discussion, we obtain the below presentation for $G/G^p\gamma_3(G)$: 
\begin{equation}\label{eq:5.1.3}
\langle x, y\ |\ x^3, y^3, [y, x, x], [y, x, y]\rangle.
\end{equation}
\noindent We now proceed to give a family of derivations from $G/G^p\gamma_3(G)\rightarrow \Omega_1(Z(G^p\gamma_3(G))\cap Z_3(G))$. Let $F$ be a free group on $\{x, y\}$. Then $\Omega_1(Z(G^p\gamma_3(G))\cap Z_3(G))$ is an $F$-module. By taking $H=F$ and $N=G^p\gamma_3(G)$ in Lemma \ref{L:4.4}, we see that the map
\begin{align*}
\tau: \Omega_1(&Z(G^p\gamma_3(G))\cap Z_3(G))^2\rightarrow \Omega_1(Z(G))^4\\
(a, b)\mapsto& \big([b, x, x][y, a, x][y, x, a],\ [b, x, y][y, a, y][y, x, b],\ [a, x, x],\ [b, y, y]\big)
\end{align*}
\noindent is a homomorphism. Let $(a, b)\in \Ker(\tau)$. By Lemma \ref{L:2.1} $(i)$, the assignment $x\mapsto a$, $y\mapsto b$ extends to a derivation $\delta_{a, b}$ of $F$. We now check that $\delta_{a, b}$ preserves the relations in \eqref{eq:5.1.3}. By Lemma \ref{L:4.5} $(ii)$, we have
\begin{align*}
\delta_{a, b}([y, x, x])= &[b, x, x][y, a, x][y, x, a]\ \text{and}\\
\delta_{a, b}([y, x, y])= &[b, x, y][y, a, y][y, x, b].
\end{align*}
\noindent Since $(a, b)\in \Ker(\tau)$, we get $\delta_{a, b}([y, x, x])=\delta_{a, b}([y, x, y])=1$. Similarly we get $\delta_{a, b}(x^p)=\delta_{a, b}(y^p)=1$ by Lemma \ref{L:4.5} $(iv)$ and $(v)$. Hence, by Lemma \ref{L:2.1} $(ii)$, $\delta_{a, b}$ induces a unique derivation from $G/G^p\gamma_3(G)\rightarrow \Omega_1(Z(G^p\gamma_3(G))\cap Z_3(G))$, and so
\begin{equation*}
|\Ker(\tau)|\le |Z^1(G/G^p\gamma_3(G), \Omega_1(Z(G^p\gamma_3(G))\cap Z_3(G)))|.
\end{equation*}
Since $d(G)=2$, we have $C_G(G^p\gamma_3(G))= Z(G^p\gamma_3(G))$ by \eqref{eq:3.11.4}. Therefore, applying Corollary \ref{cr:3.4} with $N= G^p\gamma_3(G)$, we obtain
\begin{align}\label{eq:5.1.4}
|\Ker (\tau)|\le \bigg|\frac{A^*\cap Z_4(G)}{Z(G)}\bigg|,
\end{align}
\noindent where $A^*= \{a\in Z(G^p\gamma_3(G))| a^p\in Z(G)\}$. Now we will find a lower bound for $|\Ker(\tau)|$. Since $|\Omega_1(Z(G^p\gamma_3(G))\cap Z_3(G))|\ge p^6$ by \eqref{eq:5.1.2}, and $|\im(\tau)|\le p^4$, we have
\begin{equation*}
|\Ker(\tau)|= \frac{|\Omega_1(Z(G^p\gamma_3(G))\cap Z_3(G))|^2}{|\im(\tau)|}\ge p^8.
\end{equation*}
\noindent This yields $|A^*\cap Z_4(G)|\ge p^8|Z(G)|\ge p^9$ by \eqref{eq:5.1.4}, a contradiction to $|Z_4(G)|\le p^8$, whence the proof.
\end{itemize}
\end{proof}

\begin{theorem}\label{th:5.2}
Let $p\ge 5$ and let $G$ be a finite nonabelian $p$-group of class $c$ such that $|Z_i(G)|\le p^{i+4}$ for all $i\in \{1, \ldots, c-1\}$. Then $G$ admits a non-inner automorphism of order $p$ fixing $G^p\gamma_3(G)$ elementwise, if either of the following occurs:
\begin{itemize}
\item [$(i)$] $\bigg|\frac{Z(G^p\gamma_3(G))\cap Z_4(G)}{Z(G)}\bigg|\le p^5$.
\item [$(ii)$] $\bigg|\frac{Z(G^p\gamma_3(G))\cap Z_5(G)}{Z(G)}\bigg|\le p^7$.
\end{itemize}
\end{theorem}

\begin{proof} 
Suppose that all the automorphisms of $G$ of order $p$ fixing $G^p\gamma_3(G)$ elementwise are inner. Then we assume that Hypothesis $B$ holds for $G$. Furthermore, we assume that $Z(G)$ is cyclic, $d(G)=2$, and $|\Omega_1(Z_2(G))|\le p^2$ by Theorem \ref{th:5.1}. Thus we get $|\Omega_1(Z_2(G))|=p^2$ by \eqref{eq:3.11.2}. Moreover, $G/G^p\gamma_3(G)$ is an extra-special group of order $p^3$, exponent $p$, and has presentation \eqref{eq:5.1.3}. Let $F$ be a free group on $\{x, y\}$. It follows that $Z(G^p\gamma_3(G))$ is an $F$-module.
\begin{itemize}
\item [$(i)$] We now give a family of derivations from $G/G^p\gamma_3(G)\rightarrow \Omega_1(Z(G^p\gamma_3(G))\cap Z_3(G))$. First note that the map 
\begin{align*}
\tau_1: \Omega_1&(Z(G^p\gamma_3(G))\cap Z_3(G))^2\rightarrow \Omega_1(Z(G))^2\\
(a, b)\mapsto &\big([b, x, x][y, a, x][y, x, a],\ [b, x, y][y, a, y][y, x, b]\big)
\end{align*}
\noindent is a homomorphism by Lemma \ref{L:4.4}, and let $(a, b)\in \Ker(\tau_1)$. Since $F$ is a free group, the map $x\mapsto a, y\mapsto b$ extends to a derivation $\delta_{1 a, b}$ of $F$, and by Lemma \ref{L:4.5} $(ii)$ and $(iv)$, we check that $\delta_{1 a, b}$ preserves the relations in \eqref{eq:5.1.3}. This implies that $\delta_{1a, b}$ induces a unique derivation from $G/G^p\gamma_3(G)\rightarrow \Omega_1(Z(G^p\gamma_3(G))\cap Z_3(G))$ by Lemma \ref{L:2.1} $(ii)$. Furthermore, since $d(G)=2$, we have $C_G(G^p\gamma_3(G))=Z(G^p\gamma_3(G))$ by \eqref{eq:3.11.4}, and thus applying \eqref{eq:3.4.1} with $N=G^p\gamma_3(G)$ yields that
\begin{equation}\label{eq:5.2.1}
\bigg|\frac{Z(G^p\gamma_3(G))\cap Z_4(G)}{Z(G)}\bigg|\ge |\Ker(\tau_1)|.
\end{equation}
\noindent On the other hand, since $|\Omega_1(Z(G^p\gamma_3(G))\cap Z_3(G))|\ge p^4$ by \eqref{eq:3.11.6}, and $|\im(\tau_1)|\le p^2$, we obtain
\begin{equation*}
|\ker(\tau_1)|= \frac{|\Omega_1(Z(G^p\gamma_3(G))\cap Z_3(G))|^2}{|\im(\tau_1)|}\ge p^6.
\end{equation*}
\noindent Now using \eqref{eq:5.2.1}, we obtain that $G$ admits a non-inner automorphism of order $p$ fixing $G^p\gamma_3(G)$ elementwise whenever $\bigg|\frac{Z(G^p\gamma_3(G))\cap Z_4(G)}{Z(G)}\bigg|\le p^5$.

\item [$(ii)$] We now assume that $\bigg|\frac{Z(G^p\gamma_3(G))\cap Z_4(G)}{Z(G)}\bigg|\ge p^6$ by $(i)$, and since $Z_4(G)\le Z_p(G)$, as explained in Remark \ref{R:3.10}, this implies that
\begin{equation}\label{eq:5.2.2}
|\Omega_1(Z(G^p\gamma_3(G))\cap Z_4(G))|\ge p^6.
\end{equation}
In order to prove $(ii)$, consider the map 
\begin{align*}
\tau_2: \Omega_1(&Z(G^p\gamma_3(G))\cap Z_4(G))^2\rightarrow \Omega_1(Z_2(G))^2\\
(a, b)\mapsto &\big([b, x, x][y, a, x] [y, x, b, x] [y, x, a, x] [y, x, a],\\
& [b, x, y][y, a, y] [y, x, b, y] [y, x, a, y] [y, x, b]\big).
\end{align*}
\noindent We have that $\tau_2$ is a homomorphism by Lemma \ref{L:4.4}. As in the proof of $(i)$, we obtain that every $(a, b)\in \Ker(\tau_2)$ determines a unique derivation from $G/G^p\gamma_3(G)\rightarrow \Omega_1(Z(G^p\gamma_3(G))\cap Z_4(G))$. Furthermore, since $d(G)=2$, we have $C_G(G^p\gamma_3(G))= Z(G^p\gamma_3(G))$ by \eqref{eq:3.11.4}, and hence applying \eqref{eq:3.4.1} with $N=G^p\gamma_3(G)$ yields that
\begin{equation}\label{eq:5.2.3}
\bigg|\frac{Z(G^p\gamma_3(G))\cap Z_5(G)}{Z(G)}\bigg|\ge |\Ker(\tau_2)|.
\end{equation}
\noindent On the other hand, since $|\im(\tau_2)|\le p^4$, using \eqref{eq:5.2.2} we get
\begin{align*}
|\Ker(\tau_2)|= \frac{|\Omega_1(Z(G^p\gamma_3(G))\cap Z_4(G))|^2}{|\im(\tau_2)|}\ge p^8.
\end{align*}
Thus \eqref{eq:5.2.3} yields that $G$ admits a non-inner automorphism of order $p$ leaving $G^p\gamma_3(G)$ elemnetwise fixed, whenever $\big|\frac{Z(G^p\gamma_3(G))\cap Z_5(G)}{Z(G)}\big|\le p^7$.
\end{itemize}
\end{proof}

The theorem below appears in \cite{ASZ}.
\begin{theorem}\label{th:5.3}(\cite[Theorem 2.4 and Theorem 2.5]{ASZ})
Let $G$ be a finite $p$-group and let $N$ and $M$ be normal subgroups of $G$. Then, for all $r, l\ge 0$, we have
\begin{itemize}
\item [$(i)$] $[N^{p^r}, M] \equiv [N, M]^{p^r}\ \bigg(\mod \prod_{k=1}^r [M,\ _{p^k}\ N]^{p^{r-k}}\bigg)$. 
\item [$(ii)$] $[N^{p^r},\ _l\ G] \equiv [N,\ _l\ G]^{p^r}\ \bigg(\mod \prod_{k=1}^r [N,\ _{p^k+l-1}\ G]^{p^{r-k}}\bigg)$.
\end{itemize}
\end{theorem}

In the next theorem, we prove that every finite nonabelian $p$-group of coclass $4$ and $5$ admits a non-inner automorphism of order $p$ for $p\ge 5$. Let us recall an elementary fact that if $G$ is a finite $p$-group of class $c$ and of coclass $r$, and if $|Z_i(G)|=p^{i+r-1}$ for an $i\in\{1, \ldots, c-1\}$, then $|Z_j(G)|=p^{j+r-1}$ and $G/Z_j(G)$ is a group of maximal class for all $j\in\{i, \ldots, c-1\}$. 

\begin{theorem}\label{th:5.4}
Let $p\ge 5$ and let $G$ be a finite nonabelian $p$-group.
\begin{itemize}
\item [$(i)$] If $G$ is of coclass $4$, then $G$ admits a non-inner automorphism of order $p$ fixing $G^p\gamma_3(G)$ elementwise.
\item [$(ii)$] If $G$ is of coclass $5$, then $G$ admits a non-inner automorphism of order $p$ fixing $G^p\gamma_4(G)$ elementwise.
\end{itemize}
\end{theorem}

\begin{proof}
\begin{itemize}
\item [$(i)$] Since the coclass of $G$ is $4$, we have $|Z_5(G)|\le p^8$, and thus Theorem \ref{th:5.2} $(ii)$ yields $(i)$. 
\item [$(ii)$] As in the proof of Theorem \ref{th:5.2}, we assume that Hypothesis $B$ holds for $G$, $Z(G)$ is cyclic, $d(G)=2$, $|\Omega_1(Z_2(G))|=p^2$, and $G/G^p\gamma_3(G)$ is an extra-special group of order $p^3$. Furthermore, by Theorem \ref{th:5.2} $(ii)$, $\bigg|\frac{Z(G^p\gamma_3(G))\cap Z_5(G)}{Z(G)}\bigg|\ge p^8$. Since the coclass of $G$ is $5$, we have $|Z_5(G)|\le p^9$, and this implies that $|Z_5(G)|=p^9$, $|Z(G)|=p$, and 
\begin{equation}\label{eq:5.3.1} 
Z_5(G)\le Z(G^p\gamma_3(G)).
\end{equation}
\noindent Moreover, $|Z_5(G)|=p^9$ implies that $G/Z_5(G)$ is of maximal class. Now we proceed to show the existence of $K\unlhd G$ such that $Z_4(G)\le K\le G^p\gamma_3(G)$ and $G/K$ is a group of maximal class and of order $p^4$. If $[G:Z_5(G)]\ge p^4$, then we have that $G/\gamma_4(G)Z_5(G)$ is a group of maximal class and of order $p^4$, and $Z_4(G)\le \gamma_4(G)Z_5(G)\le G^p\gamma_3(G)$ holds by \eqref{eq:5.3.1}. Now let $[G:Z_5(G)]\le p^3$. In this case, since $[G:G^p\gamma_3(G)]=p^3$, using \eqref{eq:5.3.1} we get $Z_5(G)=Z(G^p\gamma_3(G))=G^p\gamma_3(G)$, and $G^p\gamma_3(G)$ is abelian. Furthermore, since $|Z_5(G)|=p^9$, we obtain that $|G|=p^{12}$ and the class of $G$ is $7$. Note that, as the class of $G$ is $7$, $Z_4(G)\lneq Z_5(G)=G^p\gamma_3(G)$. Hence, there exists $Z_4(G)\le K\le G^p\gamma_3(G)$ with $[G^p\gamma_3(G):K]=p$. Now let us note that $G^p\le Z_3(G)\le K$. Since $\gamma_8(G)=1$, by Theorem \ref{th:5.3} $(ii)$ we obtain $[G^p,\ _3\ G]=\gamma_4(G)^p$, and taking $N=G$, $M=\gamma_3(G)$ in Theorem \ref{th:5.3} $(i)$ yields that $[G^p,\gamma_3(G)]=\gamma_4(G)^p$. We have $[G^p, \gamma_3(G)]=1$ as $G^p\gamma_3(G)$ is abelian, and hence we get $G^p\le Z_3(G)$. Moreover, $\gamma_8(G)=1$ implies that $\gamma_4(G)\le Z_4(G)\le K$. Thus $[G^p\gamma_3(G), G]\le G^p\gamma_4(G)\le K$ yielding $K\unlhd G$, and $|G/K|=p^4$. Furthermore, $G^p\le K$ implies that $\gamma_3(G)\not\le K$, and hence $G/K$ is of maximal class. This proves the existence of $K$ in either case as required. Noting that $Z_4(G)\le Z(G^p\gamma_3(G))$ by \eqref{eq:5.3.1}, $Z_4(G)\le K\le G^p\gamma_3(G)$ implies $Z_4(G)\le Z(K)$, and thus $Z_4(G)$ is a $G/K$-module. Now we proceed to give a family of derivations from $G/K\rightarrow \Omega_1(Z_4(G))$. The isomorphism class given by isomorphism type $12$ in Huppert's classification of finite $p$-groups of order $p^4$ \cite[Chapter 3 p.\ 346]{BH} is the only isomorphism class for finite $p$-groups of maximal class and order $p^4$, $p\ge 5$. Hence $G/K$ has a presentation 
\begin{equation}\label{eq:5.3.2}
\langle x, y\ |\ x^p, y^p, [y, x, y], [y, x, x, y], [y, x, x, x]\rangle.
\end{equation}
\noindent Let $F$ be a free group on $\{x, y\}$. Then $\Omega_1(Z_4(G))$ is an $F$-module. By Lemma \ref{L:4.4}, the map 
\begin{align*}
\tau_3: \Omega_1(&Z_4(G))^2\rightarrow \Omega_1(Z_2(G))\times \Omega_1(Z(G))^2\\
(a, b)\mapsto &\big([b, x, y][y, a, y][y, x, b, y][y, x, a, y][y, x, b],\ [\nu(a, b), y],\ [\nu(a, b), x]\big)
\end{align*}
is a homomorphism, where $\nu(a, b)= [b, x, x][y, a, x][y, x, b, x][y, x, a, x][y, x, a]$. Let $(a, b)\in \Ker(\tau_3)$. Since $F$ is a free group, the map $x\mapsto a$, $y\mapsto b$ extends to a derivation $\delta_{3_{a, b}}$ of $F$. Now we proceed to check that $\delta_{3 a, b}$ preserves the relations in \eqref{eq:5.3.2}. Using Lemma \ref{L:4.5} $(ii)$, we get $\delta_{3 a, b}([y, x, y])= [b, x, y][y, a, y][y, x, b, y][y, x, a, y][y, x, b]=1$. Similarly $\delta_{3 a, b}([y, x, x])= \nu(a, b)$, and so, by Lemma \ref{L:4.5} $(iii)$, we obtain $\delta_{3 a, b}([y, x, x, y])= [\nu(a, b), y]=1$ and $\delta_{3 a, b}([y, x, x, x])= [\nu(a, b), x]=1$. By Lemma \ref{L:4.5} $(iv)$ we get $\delta_{3 a, b}(x^p)= \delta_{3 a, b}(y^p)=1$. Hence $\delta_{3 a, b}$ induces a derivation on $G/K$ by Lemma \ref{L:2.1} $(ii)$, which we again denote with $\delta_{3 a, b}$. Let $\alpha_{3 a, b}=\varphi(\delta_{3 a, b})\in C_{\Aut(G)}(G/K, K)$. It follows that $\alpha_{3 a, b}$ has order $p$, and fixes $K$ and $G/\Omega_1(Z_4(G))$ elmentwise. Since the class of $G/K$ is $3$, $G/K$ is a regular group, and since $x^p, y^p\in K$ by \eqref{eq:5.3.2}, we obtain $G^p\le K$. Furthermore, $\gamma_4(G)\le K$ as $|G/K|=p^4$, and thus $\alpha_{3 a, b}$ fixes $G^p\gamma_4(G)\le K$ elementwise. Suppose $\alpha_{3 a, b}=i_{u_{3 a, b}}$ is an inner automorphism of $G$, then $\alpha_{3 a, b}|_{G/\Omega_1(Z_4(G))}=id$ implies that $u_{3 a, b}\in Z_5(G)$. Hence if $\alpha_{3 a, b}$ is an inner automorphism of $G$ for all $(a, b)\in \Ker(\tau_3)$, then we obtain
\begin{equation}\label{eq:5.3.3}
|\Ker(\tau_3)|\le \bigg|\frac{Z_5(G)}{Z(G)}\bigg|=p^8.
\end{equation}
Now we conclude the proof by contradicting \eqref{eq:5.3.3}. By \eqref{eq:5.3.1}, $Z_4(G)\le Z(G^p\gamma_3(G))$, and hence using Theorem \ref{th:5.2} $(i)$ we obtain $\big|\frac{Z_4(G)}{Z(G)}\big|\ge p^6$. As explained in the Remark \ref{R:3.10}, this implies that
\begin{equation}\label{eq:5.3.4}
|\Omega_1(Z_4(G))|\ge p^6.
\end{equation}
\noindent Now we find an upper bound for $|\im(\tau_3)|$. Let us consider the map 
\begin{align*}
\mu: \Omega_1(&Z_2(G))\rightarrow \Omega_1(Z(G))^2\\
w\mapsto& ([w, y], [w, x]).
\end{align*}
We have that $\mu$ is a homomorphism by Lemma \ref{L:4.4}. Noting that $\Ker(\mu)=\Omega_1(Z(G))$, and since $|\Omega_1(Z_2(G))|=p^2$, we get $|\im(\mu)|=p$. Since $\nu(a, b)\in \Omega_1(Z_2(G))$ for all $a, b\in \Omega_1(Z_4(G))$, we obtain that $\im(\tau_3)\le \Omega_1(Z_2(G))\times \im(\mu)$. Thus $|\im(\tau_3)|\le |\Omega_1(Z_2(G))| |\im(\mu)|=p^3$, and using \eqref{eq:5.3.4} we get
\begin{equation*}
|\Ker(\tau_3)|\ge \frac{|\Omega_1(Z_4(G))|^2}{|\im(\tau_3)|}\ge p^9.
\end{equation*}
\end{itemize}
\end{proof}

\section*{Acknowledgements}
I would like to thank my advisor Viji Z Thomas for the careful reading of the manuscript and his valuable suggestions and constant support.

\bibliographystyle{amsplain}
\bibliography{outbibliography}
\end{document}